\documentclass[11pt]{amsproc}
 \usepackage[margin=1in]{geometry}
\usepackage{setspace,fullpage}
\usepackage{graphicx}
\usepackage[nice]{nicefrac}
\usepackage{amssymb}

\usepackage{amsmath,amsthm,amscd,amssymb, mathrsfs}

\usepackage{latexsym}

\numberwithin{equation}{section}

\theoremstyle{plain}
\newtheorem{theorem}{Theorem}[section]
\newtheorem{lemma}[theorem]{Lemma}
\newtheorem{corollary}[theorem]{Corollary}

\theoremstyle{definition}

\theoremstyle{remark}

\newtheorem{case[theorem]}{Case}

\def\x{{\bf x}}

\title[\parbox{14cm}{\centering{  Sharp $L^p\to L^r$ estimates for $k$-plane transforms \hspace{1in}}} \quad]{ Sharp $L^p\to L^r$ estimates for $k$-plane transforms in finite fields}
\author{ Doowon Koh and Dongyoon Kwak}

\address{Department of Mathematics\\
Chungbuk National University \\
Cheongju Chungbuk 28644, Korea}
\email{koh131@chungbuk.ac.kr}

\address{Department of Mathematics\\
Chungbuk National University \\
Cheongju Chungbuk 28644, Korea}
\email{yoon0506@chungbuk.ac.kr}

\thanks{
This research was supported by Basic Science Research Program through the National Research Foundation of Korea(NRF) funded by the Ministry of Education, Science and Technology(NRF-2015R1A1A1A05001374)}


\begin{document}

\begin{abstract} We study mapping properties of  finite field $k$-plane transforms.
Using geometric combinatorics, we do an elaborate analysis to recover the critical endpoint estimate.
As a consequence, we obtain optimal $L^p\to L^r$ estimates for all $k$-plane transforms in the finite field setting. In addition, applying 
H\"{o}lder's inequality to our results, we obtain an estimate for multilinear $k$-plane transforms. 
\end{abstract}
\maketitle

\section{Introduction}

Over the last few decades,  finite field analogs of  Euclidean harmonic analysis problems have been extensively studied.
Tom Wolff \cite{Wo} initially proposed the finite field Kakeya problem which was solved by Dvir \cite{Dv} using the polynomial method.
Adapting the method, Ellenberg, Oberlin, and Tao \cite{EOT} settled the finite field Kakeya maximal conjecture.
In 2004, the finite field restrict problem was also initiated by Mockenhaupt and Tao \cite{MT04}.
Like Euclidean case, the finite field restriction conjectures are still open  although some progress on this problem has been made by researchers (see, for example, \cite{Le14, LL13, Ko16, LL10, KS12, IK10}).
We also refer the reader to Wright's lecture note \cite{Wr13} for finite ring restriction problems. \\

After Mockenhaupt and Tao,  the finite field (maximal) averaging problem was formulated and studied by Carbery, Stones, and Wright \cite{CSW08}.
In the paper, they also initially studied mapping properties of finite field $k$-plane transforms. The main purpose of this paper is to give the complete answer to the boundedness problem on $k$-plane transforms in the finite field setting.
Let us review the definition and notation related to finite field $k$-plane transforms. Let $\mathbb F_q$ be a finite field with $q$ elements.
We denote by $\mathbb F_q^d, d\ge 2,$ a $d$-dimensional vector space over the finite field $\mathbb F_q$. We endow $\mathbb F_q^d$ with a normalized counting measure $d\x$ so that for $f: \mathbb F_q^d \to \mathbb C$ we have
$$ \|f\|_{L^p(\mathbb F_q^d, d\x)} =  \left\{\begin{array}{ll}  \left(q^{-d}\sum\limits_{\x\in \mathbb F_q^d} |f(\x)|^p\right)^{\frac{1}{p}}\quad &\mbox{if} \quad 1\le p<\infty\\
                                                  \max\limits_{\x\in \mathbb F_q^d} |f(\x)| \quad & \mbox{if} \quad p=\infty. \end{array}\right.$$

Given a fixed dimension $d\ge 2,$ let us choose  $1\le k \le d-1$ and denote by $M_k$ the set of all $k$-planes in $\mathbb F_q^d$, which means affine subspaces in $\mathbb F_q^d$ with dimension $k.$
From basic linear algebra, we notice that
$$ |M_k|\sim q^{(d-k)(k+1)}.$$
Moreover, if $\Pi_{k,s}$ denotes the number of $k$-planes containing a given $s$-plane with $0\le s \le k,$  then
$$ |\Pi_{k,s}|\sim q^{(d-k)(k-s)}.$$
Throughout this paper, for $X, Y>0$, we use $X\lesssim Y$ if there is a constant $C>0$ independent of $q$ such that $X\le CY$,
and $X\sim Y$ if $X\lesssim Y$ and $Y\lesssim X.$
We now endow $M_k$ with a normalized counting measure $\lambda_k$ so that for $g: M_k \to \mathbb C,$ we define its integral as
$$ \int_{M_k} g(\omega)~d\lambda_k(\omega) =\frac{1}{|M_k|} \sum_{\omega\in M_k} g(\omega),$$
where $|M_k|$ denotes the cardinality of the set $M_k.$
With the above notation, we define the $k$-plane transform $T_k$ of a function $f: \mathbb F_q^d \to \mathbb C$ as
$$ T_kf(\omega)=\int_{\omega} f(\x)~ d\sigma_\omega (\x) := \frac{1}{|\omega|} \sum_{\x\in \omega} f(\x),$$
where $d\sigma_\omega$ denotes the normalized surface measure on the $k$-plane $\omega \in M_k.$ In particular, the operator $T_k$ is called  the $X$-ray transform for $k=1$ and  the Radon transform for $k=d-1.$
In this finite field setting, the $k$-plane transform problem asks us to determine  exponents $1\le p, r\le \infty$ such that
the estimate
\begin{equation}\label{defT}\|T_kf\|_{L^r(M_k,d\lambda_k)} \lesssim \|f\|_{L^p(\mathbb F_q^d,d\x)}\end{equation}
holds for every function $f: \mathbb F_q^d \to \mathbb C,$
where the operator norm of $T_k$ must be independent of $q$ which is the size of the underlying finite field $\mathbb F_q.$
In the Euclidean setting, as a consequence of mixed norm estimates for the $k$-plane transform, this problem was completely solved by M. Christ \cite{Ch84} who improved on results of Drury \cite{Dr84}.
On the other hand,  Carbery, Stones, and Wright \cite{CSW08} obtained sharp restricted type estimates for all $k$-plane transforms in finite fields.
In fact, they proved that for the critical endpoint
$(1/p, 1/r)=((k+1)/(d+1), 1/(d+1))$, the estimate \eqref{defT} holds for all characteristic functions $f=\chi_E$ on any set $E\subset \mathbb F_q^d.$
More formally, they obtained the following result.
\begin{theorem}\label{Car} Let $d\geq 2$ and $1\leq k\leq d-1.$ If the estimate
\begin{equation}\label{CSW}\|T_kf\|_{L^r(G_k,d\lambda_k)} \lesssim \|f\|_{L^p(\mathbb F_q^d,d\x)}\end{equation}
holds for all functions $f$ on $\mathbb F_q^d$, then $(1/p, 1/r)$ lies on the convex hull $H$ of
$$((k+1)/(d+1), 1/(d+1)), (0,0),(1,1)~~\mbox{and}~~(0,1).$$
Conversely, if $(1/p, 1/r)$ lies in $H\setminus ((k+1)/(d+1), 1/(d+1)),$ then \eqref{CSW} holds for all functions $f$  on $\mathbb F_q^d.$ Furthermore,  the restricted type inequality
\begin{equation}\label{resWeak} \|T_kf\|_{L^{d+1}(M_k,d\lambda_k)} \lesssim
\|f\|_{L^{\frac{d+1}{k+1}, 1}(\mathbb F_q^d,d\x)}\end{equation}
holds for all functions $f$ on $\mathbb F_q^d.$\end{theorem}

\subsection{Statement of the main result} The first part of Theorem \ref{Car} states the necessary condition for the boundedness of the $k$-plane transform. Note that the necessary condition would be in fact sufficient if the restricted type estimate \eqref{resWeak} can be extended to a strong type estimate. Therefore, to settle the finite field $k$-plane transform problem,
we only need to establish the strong type $L^{(d+1)/(k+1)} \to L^{d+1}$ estimate. In \cite{Ko13}, it was already proved that the strong type estimate holds for the $X$-ray transform ($k=1$) and the Radon transform ($k=d-1$).
To obtain the sharp estimate, methods of the discrete Fourier analysis and geometric combinatorics were used for the Radon transform and the $X$-ray transform, respectively.
In this paper, we extend the work to other $k$-plane transforms so that we obtain full solution of the finite field $k$-plane transform problem.
The main result we shall prove is as follows.
\begin{theorem}\label{newmain1} Let $d\geq 2$ and $1\leq k\leq (d-1).$  Then we have
$$  \|T_kf\|_{L^{d+1}(M_k,d\lambda_k)} \lesssim \|f\|_{L^{\frac{d+1}{k+1}}(\mathbb F_q^d,d\x)} \quad\mbox{for all}~~f ~~\mbox{on}~~\mathbb F_q^d.$$
\end{theorem}

As mentioned before, this theorem for $k=1$ and $k=d-1$ was already obtained in \cite{Ko13} and thus our main result is new for the case $ 2\le k\le d-2.$
It is also known that  one can deduce the result of Theorem \ref{newmain1} for $k=1$ (the $X$-ray transform) by applying the finite field Kakeya maximal conjecture which was solved by Ellenberg, Oberlin, and Tao (see Theorem 1.3 and Remark 1.4 in \cite{EOT}). Likewise one could also derive the results of Theorem \ref{newmain1} for $ 2\le k\le d-1$ if one could prove conjecture on $k$-plane maximal operator estimates in finite fields (see Conjecture 4.13 in \cite{EOT}). However, the conjecture has not been solved (see \cite{Bu} for the best known result on this problem).

\vskip0.25in
By repeatedly using H\"{o}lder's inequality, the following estimate for a multilinear $k$-plane transform can be deduced from Theorem  \ref{newmain1}.
\begin{corollary}\label{corfun} With the assumption of Theorem \ref{newmain1}, we have
$$ \left\|\prod_{j=1}^{d+1} T_kf_j\right\|_{L^{1}(M_k,d\lambda_k)} \lesssim \,\,\prod_{j=1}^{d+1}\|f_j\|_{L^{\frac{d+1}{k+1}}(\mathbb F_q^d,d\x)}$$
for all functions $f_j, j=1,2,\ldots, (d+1),$ on $\mathbb F_q^d.$
\end{corollary}
\begin{proof}
Since $1=\sum_{t=1}^{d+1} \frac{1}{d+1},$  if we repeatedly use H\"{o}lder's inequality, we see that
$$\left\|\prod_{j=1}^{d+1} T_kf_j\right\|_{L^{1}(M_k,d\lambda_k)} 
\le \,\,\prod_{j=1}^{d+1} \|T_kf_j\|_{L^{d+1}(M_k,d\lambda_k)},$$
and so the statement of corollary follows immediately from Theorem \ref{newmain1}. 
\end{proof}
Taking $f=f_j$ for $j=1,2,\ldots, (d+1)$, notice that Corollary \ref{corfun} also implies Theorem \ref{newmain1}.
It would be interesting to find another proof of Corollary \ref{newmain1}  
(see, for example, \cite{BCW, Ca}).

\section{Proof of the main theorem (Theorem \ref{newmain1})}

We start proving Theorem \ref{newmain1} by making certain reductions. We aim to prove for each integer $1\leq k\leq (d-1)$ that the estimate
\begin{equation}\label{aim1} \|T_kf\|_{L^{d+1}(M_k,d\lambda_k)} \lesssim \|f\|_{L^{\frac{d+1}{k+1}}(\mathbb F_q^d,d\x)} = \left(q^{-d} \sum_{\x\in \mathbb F_q^d} |f(\x)|^{\frac{d+1}{k+1}}\right)^{\frac{k+1}{d+1}}\end{equation}
holds for all functions $f$ on $\mathbb F_q^d,$
where we recall that $M_k$ denotes the collection of all affine $k$-planes in $\mathbb F_q^d.$ Without loss of generality, we may assume that
$f$ is a non-negative real-valued function and
\begin{equation}\label{easy1} \sum_{\x\in \mathbb F_q^d} f(\x)^{\frac{d+1}{k+1}}=1.\end{equation}
Thus, we also assume that $\|f\|_\infty \leq 1.$ Furthermore, we may assume that $f$ is written by a step function
\begin{equation}\label{easy2} f(\x)=\sum_{i=0}^\infty 2^{-i} E_i(\x),\end{equation}
where $E_i^{\prime}$s are disjoint subsets of $\mathbb F_q^d$ and we write $E(\x)$ for the characteristic function $\chi_E$ on a set $E\subset \mathbb F_q^d,$ which allows us to use a simple notation. From (\ref{easy1}) and (\ref{easy2}), we also assume that
\begin{equation} \label{easy3} \sum_{j=0}^\infty 2^{-\frac{(d+1)j}{k+1}}|E_j|=1\quad\mbox{and so} ~~ |E_j|\leq 2^{\frac{(d+1)j}{k+1}}~~\mbox{for all}~~j=0,1,\cdots.\end{equation}
Thus, to prove \eqref{aim1}, it suffices to prove that
\begin{equation}\label{suff0} \|T_{k}f\|^{d+1}_{L^{d+1}( M_k, d\lambda_k)} \lesssim q^{-d(k+1)},\end{equation}
for all functions $f$ such that the conditions \eqref{easy2}, \eqref{easy3} hold. Since we have assumed that $f\geq 0$, it is clear that $T_{k}f$ is also a non-negative real-valued function on $M_k.$ By expanding the left hand side of the above inequality \eqref{suff0} and using the facts that $|\omega|=q^k$ for $\omega\in M_k$ and $|M_k|\sim q^{(d-k)(k+1)}$, we see that
$$ \|T_{k}f\|^{d+1}_{L^{d+1}( M_k, d\lambda_k)} = \frac{1}{|M_k|} \sum_{\omega\in M_k} \left(T_{k}f(\omega)\right)^{d+1} $$
$$\sim \frac{1}{q^{k(d+1)}} \frac{1}{q^{(d-k)(k+1)}} \sum_{i_0=0}^\infty  \dots \sum_{i_d=0}^\infty 2^{-(i_0+\dots +i_{d})} \sum_{(\x^0,\dots,\x^{d})\in E_{i_{0}}\times \dots \times E_{{i_{d}}}} \sum_{\omega\in M_k} \omega(\x^0)\dots \omega(\x^{d})$$
$$\sim \frac{1}{q^{k(d+1)}} \frac{1}{q^{(d-k)(k+1)}} \sum_{0=i_0\leq i_1\leq \dots\leq i_{d}<\infty}   2^{-(i_0+\dots +i_{d})} \sum_{(\x^0,\dots,\x^{d})\in E_{i_{0}}\times \dots \times E_{i_{d}}} \sum_{\omega\in M_k} \omega(\x^0)\dots \omega(\x^{d}),$$
where the last line follows from  the symmetry of $i_0, \cdots,i_{d}.$
Now, we decompose the sum over $(\x^0,\dots,\x^{d})\in E_{i_{0}}\times \dots \times E_{i_{d}}$ as

$$\sum_{(\x^0,\dots, \x^{d})\in E_{i_{0}}\times \dots \times E_{i_{d}}}=\sum_{s=0}^\infty \sum_{(\x^0,\dots,\x^{d})\in \Delta(s,i_0,\dots,i_{d})},$$
where $\Delta(s, i_0,\dots,i_{d}):=\{(\x^0,\dots,\x^{d})\in E_{i_{0}}\times \dots \times E_{i_{d}}: [\x^0,\dots,\x^{d}]~\mbox{is a}~s\mbox{-plane}\}$ and $[\x^0,\dots, \x^{d}]$ denotes the smallest affine subspace containing the elements $\x^0,\dots, \x^{d}.$ Now, notice that if $ s>k$ and  $(\x^0,\dots, \x^{d}) \in \Delta(s,i_0,\dots,i_{d}),$  then the sum over $\omega\in M_k$ vanishes. On the other hand, if $0\leq s\leq k,$ then the sum over $\omega\in M_k$
is same as the number of $k$-planes containing the unique $s$-plane, that is $\sim q^{(d-k)(k-s)}.$ From these observations and (\ref{suff0}), our task is to show that for all $E_i, i=0,1,\dots,$ satisfying the condition (\ref{easy3}),

\begin{equation}\label{suffmain} \sum_{i_0=0}^\infty \sum_{i_1\geq i_0}^\infty \cdots \sum_{i_{d}\geq i_{d-1}}^\infty 2^{-(i_0+i_1+\cdots +i_{d})} \sum_{s=0}^k |\Delta(s,i_0,\dots,i_{d})| q^{-s(d-k)} \lesssim 1.
\end{equation}

In \cite{Ko13}, it was shown that this inequality holds true for a simple case $k=1$, and so the sharp estimate for the $X$-ray transform was obtained.
However, when $k\ge 2$ and the dimension $d$ becomes bigger, it is not a simple problem to prove \eqref{suffmain}, because a lot of complicate cases happen in finding an upper bound of $|\Delta(s, i_0, i_1, \ldots, i_d)|.$
In the following subsections, we shall prove \eqref{suffmain} by making further reductions so that the proof of Theorem \ref{newmain1} will be complete.
\subsection{Proof of \eqref{suffmain}}
For each $s=0,1,\ldots, k$, it suffices to prove that
\begin{equation}\label{suffmain1} \sum_{i_0=0}^\infty \sum_{i_1\geq i_0}^\infty \cdots \sum_{i_{d}\geq i_{d-1}}^\infty 2^{-(i_0+i_1+\cdots +i_{d})}
|\Delta(s,i_0,\dots,i_{d})| q^{-s(d-k)} \lesssim 1.
\end{equation}
First fix $s=0,1,\ldots,k$ and the sets $E_{i_0}, E_{i_1}, \ldots, E_{i_d}.$
To find an upper bound of $|\Delta(s,i_0,\dots,i_{d})|$, we shall decompose the set $\Delta(s,i_0,\dots,i_{d})$ as a union of its disjoint subsets. For each $(\x^0, \x^1,\ldots, \x^d)\in  \Delta(s,i_0,\dots,i_{d})$ there are unique nonnegative integers $\ell_0, \ell_1, \ldots, \ell_s$ with
$0=\ell_0< \ell_1< \ell_2<\ldots<\ell_{s-1} < \ell_s \le d$ such that
$\x^0, \x^1,\ldots, \x^{\ell_j}$ determine a $j$-plane and $\x^0, \x^1, \ldots, \x^{\ell_{j}-1}$ determine a $(j-1)$-plane for $j=1,2,\dots, s,$ where we define $\ell_0=0.$ Therefore, we can write
$$ \Delta(s,i_0,\dots,i_{d})= \bigcup_{0=\ell_0< \ell_1<\ell_2<\dots <\ell_s\le d} L(s,i_0,\dots,i_{d},\ell_0, \ell_1, \ldots, \ell_s),$$
where $L(s,i_0,\dots,i_{d}, \ell_0, \ell_1, \ldots, \ell_s)$ consists of those members $(\x^0, \x^1,\ldots, \x^d) \in \Delta(s,i_0,\dots,i_{d})$ such that for every $j=1,2, \ldots, s$, the affine span of $\x^0, \x^1, \ldots, \x^{\ell_j}$ is of dimension $j$ and the affine span of $\x^0, \x^1, \ldots, \x^{\ell_{j}-1}$ is of dimension $j-1.$
Now, let us find an upper bound of $|L(s,i_0,\dots,i_{d},\ell_0, \ell_1, \ldots, \ell_s)|$ where $0=\ell_0< \ell_1<\dots <\ell_s\le d.$ It is clear that if $(\x^0, \x^1,\ldots, \x^d)\in L(s,i_0,\dots,i_{d},\ell_0, \ell_1, \ldots, \ell_s)$, then there are at most $|E_{i_{\ell_j}}|$ choices for $\x^{\ell_j}, j=0, 1, \ldots, s.$ In addition, if $\ell_{j} <t<\ell_{j+1}$, \footnote{Throughout this paper we shall assume that $\ell_{s+1}=d+1.$} then there are at most $\min\{|E_{i_t}|, q^{j}\}$ choices for $\x^t,$ because the point $\x^t$ must be contained in the affine $j$-plane of points $\x^0, \x^1, \ldots, \x^{\ell_j}$, otherwise $t$ would be greater than or equal to $\ell_{j+1}$ by the definition of $L(s,i_0,\dots,i_{d},\ell_0, \ell_1, \ldots, \ell_s).$ From these observations, it follows that
\begin{equation}\label{prod}|L(s,i_0,\dots,i_{d},\ell_0, \ell_1, \ldots, \ell_s)|\le \prod_{j=0}^s \left( |E_{i_{\ell_j}}|\prod_{t=\ell_j+1}^{\ell_{j+1}-1} \min\{|E_{i_t}|, q^j\}\right),\end{equation}
where we define that if $\ell_{j+1}=\ell_j+1,$ then
$$\prod_{t=\ell_j+1}^{\ell_{j+1}-1} \min\{|E_{i_t}|, q^j\}=1.$$
Let $A=\{j\in \{0,1,\ldots,s\}: \ell_{j+1}\ne \ell_j+1\}.$ Since $\sum_{j\in A} (\ell_{j+1}-\ell_j-1)=d-s\ge d-k$, the right hand side of \eqref{prod} has at least $(d-k)$ factors each of which takes a form $\min\{|E_{i_t}|, q^j\}$ for some $j\in A$ and $t$ with $\ell_j+1\le t \le \ell_{j+1}-1.$
Now, we estimate that $\min\{|E_{i_t}|, q^j\} \le q^j$ for the $(d-k)$ largest numbers in the set of such $t$,
and $\min\{|E_{i_t}|, q^j\} \le |E_{i_t}|$ for the rest $(k-s)$ numbers $t.$ By this way, we can obtain an upper bound of the right hand side of \eqref{prod} which we shall denote by $U(s, i_0, \ldots, i_d, \ell_0, \ell_1, \ldots, \ell_s).$ For example, if $d=7, k=4, s=2, \ell_0=0, \ell_1=2, \ell_2=5,$ then
$$ U(s, i_0, \ldots, i_d, \ell_0, \ell_1, \ldots, \ell_s)=|E_{i_0}||E_{i_1}||E_{i_2}| |E_{i_3}| q |E_{i_5}| q^2 q^2.$$
It is clear that
\begin{align*}
 |\Delta(s, i_0, \ldots, i_d)| &\lesssim
 \max_{0=\ell_0<\ell_1 <\cdots <\ell_s\le d} |L(s, i_0, \ldots, i_d, \ell_0, \ell_1, \ldots, \ell_s)|\\
 & \le \max_{0=\ell_0<\ell_1 <\cdots <\ell_s\le d} U(s, i_0, \ldots, i_d, \ell_0, \ell_1, \ldots, \ell_s).
\end{align*}

Thus, to prove the estimate \eqref{suffmain1}, it is enough to show that for every $s=0,1, \ldots, k$ and $0=\ell_0 <\ell_1 <\cdots <\ell_s\le d,$
\begin{equation} \label{reduce1}
\sum_{i_0=0}^\infty \sum_{i_1\geq i_0}^\infty \cdots \sum_{i_{d}\geq i_{d-1}}^\infty 2^{-(i_0+i_1+\cdots +i_{d})}
 ~U(s, i_0, \ldots, i_d, \ell_0, \ell_1, \ldots, \ell_s)~q^{-s(d-k)} \lesssim 1.
\end{equation}
We claim that it suffices to prove this estimate \eqref{reduce1} only for the case when $s=k.$
This claim follows by observing from the definition of $U$ that
given a value $U(s, i_0, \ldots, i_d, \ell_0, \ell_1, \ldots, \ell_s)$ for $s=0,1,\ldots, (k-1),$
we can choose numbers $\ell_1', \ell_2', \ldots, \ell_{s+1}'$ with $1\le \ell_1'< \ell_2'< \ldots, \ell_{s+1}'\le d$ such that
$$ U(s, i_0, \ldots, i_d, \ell_0, \ell_1, \ldots, \ell_s) ~q^{d-k} = U(s+1, i_0, \ldots, i_d, \ell_0, \ell_1', \ell_2', \ldots, \ell_{s+1}').$$
In fact, $\{\ell_1', \ell_2', \ldots, \ell_{s+1}'\}$ can be selected by adding one number, say $\ell'$, to
$\{ \ell_1, \ldots, \ell_s\},$ where $\ell'=\ell_{j_0} +1$ and  $j_{0}$ is defined by
$$j_0 =\min\{j\in \{0,1,\ldots,s\}: \ell_{j+1} \ne \ell_j +1\}.$$

Therefore, our final task is to prove that for every nonnegative integers $\ell_0, \ell_1, \ldots, \ell_k$ with $ 0=\ell_0 < \ell_1 < \cdots < \ell_k\le d,$ we have
\begin{equation}\label{reduce2}
S:=\sum_{i_0=0}^\infty \sum_{i_1\geq i_0}^\infty \cdots \sum_{i_{d}\geq i_{d-1}}^\infty 2^{-(i_0+i_1+\cdots +i_{d})}
 ~U(k, i_0, \ldots, i_d, \ell_0, \ell_1, \ldots, \ell_k)~q^{-k(d-k)} \lesssim 1.
\end{equation}
This shall be proved in the following subsection.
\subsection{ Proof of the estimate \eqref{reduce2}}
We begin with a preliminary lemma.
\begin{lemma} \label{lemU}
With the notation above, we have
\begin{align*}
q^{-k(d-k)}~U(k, i_0, \ldots, i_d, \ell_0, \ell_1, \ldots, \ell_k)=\left( \prod_{t=0}^k |E_{i_{\ell_t}}|   \right) \left( q^{-\sum\limits_{t=1}^k (\ell_t-t)}\right).
\end{align*}
\end{lemma}
\begin{proof}
By the definition of $U$, we see that
\begin{align*} U(k, i_0, \ldots, i_d, \ell_0, \ell_1, \ldots, \ell_k)&= \prod_{t=0}^k |E_{i_{\ell_t}}|~ q^{t(\ell_{t+1}-\ell_t-1)}\\
&=\left( \prod_{t=0}^k |E_{i_{\ell_t}}|   \right) \left(q^{\sum\limits_{t=0}^k t(\ell_{t+1}-\ell_t-1)}\right),\end{align*}
where  $\ell_0=0$ and $\ell_{k+1}=d+1.$ It follows that
$$q^{-k(d-k)}~U(k, i_0, \ldots, i_d, \ell_0, \ell_1, \ldots, \ell_k)=\left( \prod_{t=0}^k |E_{i_{\ell_t}}|   \right) \left( q^{-k(d-k)+\sum\limits_{t=1}^k t(\ell_{t+1}-\ell_t-1)}\right).$$
Thus, the proof of Lemma \ref{lemU} will be complete  if we show that
$$  {-k(d-k)+\sum\limits_{t=1}^k t(\ell_{t+1}-\ell_t-1)} = -\sum\limits_{t=1}^k (\ell_t-t).$$
To prove  this equality, observe that
$$\sum\limits_{t=1}^k t(\ell_{t+1}-\ell_t) =-\ell_1-\ell_2-\cdots -\ell_k + k\ell_{k+1} =\left(-\sum_{t=1}^k \ell_t \right) +k(d+1).$$
Then we obtain that
\begin{align*}-k(d-k)+\sum\limits_{t=1}^k t(\ell_{t+1}-\ell_t-1) &= -k(d-k) - \sum_{t=1}^k \ell_t  +k(d+1) -\frac{k(k+1)}{2}\\
&= \frac{k(k+1)}{2} -\sum_{t=1}^k \ell_t = \sum_{t=1}^k t -\sum_{t=1}^k \ell_t = -\sum_{t=1}^k (\ell_t-t).\end{align*}
\end{proof}

We shall give the complete proof of the estimate \eqref{reduce2}.
From Lemma \ref{lemU}, we aim to prove that
$$ S=\sum_{i_0=0}^\infty \sum_{i_1\geq i_0}^\infty \cdots \sum_{i_{d}\geq i_{d-1}}^\infty 2^{-(i_0+i_1+\cdots +i_{d})} \left( \prod_{t=0}^k |E_{i_{\ell_t}}|   \right) \left( q^{-\sum\limits_{t=1}^k (\ell_t-t)}\right)
 ~\lesssim 1.$$
Write the term $S$ as
\begin{align*}S= \sum_{i_0=0}^\infty \sum_{i_1=0}^\infty \cdots \sum_{i_{d}= 0}^\infty
1_{i_0\le i_1\le \cdots \le i_d}~2^{-(i_0+i_1+\cdots +i_{d})} \left( \prod_{t=0}^k |E_{i_{\ell_t}}|   \right) \left( q^{-\sum\limits_{t=1}^k (\ell_t-t)}\right),
\end{align*}
where we define that $1_{i_0\le i_1\le \cdots \le i_d}=1$ if $i_0\le i_1\le \cdots \le i_d$, and $0$ otherwise. By the Fubini theorem, we can decompose  the sums as follows.\footnote{To simplify notation, the general term was omitted.}
\begin{equation}\label{sumchange} S=\sum_{i_0=0}^\infty \sum_{i_1=0}^\infty \cdots \sum_{i_{d}= 0}^\infty
= \sum_{i_0=0}^\infty \sum_{i_{\ell_1}=0}^\infty\sum_{i_{\ell_2}=0}^\infty \cdots \sum_{i_{\ell_k}= 0}^\infty
\sum_{i_{j_1}=0}^\infty \sum_{i_{j_2}=0}^\infty \cdots \sum_{i_{j_{d-k}}=0}^\infty,\end{equation}
where $j_1,j_2, \ldots, j_{d-k}$ denote natural numbers such that
$\{j_1, j_2, \ldots, j_{d-k}\}=\{1,2,\ldots, d\}\setminus \{\ell_1, \ell_2, \ldots, \ell_k\}$ and
$1\le j_1<j_2<\cdots < j_{d-k}\le d.$
For each $j_t, ~t=1,2,\ldots, d-k$, let us denote by $<j_t>$ the greatest element of $\{\ell_0, \ell_1, \ldots, \ell_k\}$ less than $j_t.$ From the definition of $1_{i_0\le i_1\le \cdots \le i_d},$ it is clear that
\begin{align*}S&\le  \sum_{i_0=0}^\infty \sum_{i_{\ell_1}\ge i_{0}}^\infty \cdots \sum_{i_{\ell_k}\ge i_{\ell_{k-1}}}^\infty
\sum_{i_{j_1}\ge i_{<j_1>}}^\infty \sum_{i_{j_2}\ge i_{<j_2>}}^\infty \cdots \sum_{i_{j_{d-k}} \ge i_{<j_{d-k}>}}^\infty 2^{-(i_0+i_1+\cdots +i_{d})} \left( \prod_{t=0}^k |E_{i_{\ell_t}}|   \right) \left( q^{-\sum\limits_{t=1}^k (\ell_t-t)}\right)\\
&=  \sum_{i_0=0}^\infty \sum_{i_{\ell_1}\ge i_{0}}^\infty \cdots \sum_{i_{\ell_k}\ge i_{\ell_{k-1}}}^\infty
\left( \prod_{t=0}^k |E_{i_{\ell_t}}|   \right) \left( q^{-\sum\limits_{t=1}^k (\ell_t-t)}\right)
\left(\sum_{i_{j_1}\ge i_{<j_1>}}^\infty  \cdots \sum_{i_{j_{d-k}} \ge i_{<j_{d-k}>}}^\infty 2^{-(i_0+i_1+\cdots +i_{d})}\right). \end{align*}
Inner sums can be computed by using a simple fact that the value of a convergent geometric series  is similar to the first term of the series. In addition, use the definition of $<j_t>,~t=1,2,\ldots, d-k,$ and  a simple fact that there are $(\ell_{t+1}-\ell_t-1)$ natural numbers between $\ell_t$ and $\ell_{t+1}$ for $t=0,1,\ldots, k.$
We are led to the estimate
$$ \sum_{i_{j_1}\ge i_{<j_1>}}^\infty \sum_{i_{j_2}\ge i_{<j_2>}}^\infty \cdots \sum_{i_{j_{d-k}} \ge i_{<j_{d-k}>}}^\infty 2^{-(i_0+i_1+\cdots +i_{d})} \sim \prod_{t=0}^k 2^{-(\ell_{t+1}-\ell_t) i_{\ell_t}}.$$
It follows that
\begin{align} \label{allmost}
S\nonumber&\lesssim \sum_{i_0=0}^\infty \sum_{i_{\ell_1}\ge i_{0}}^\infty \cdots \sum_{i_{\ell_k}\ge i_{\ell_{k-1}}}^\infty \left( \prod_{t=0}^k |E_{i_{\ell_t}}|   \right) \left( q^{-\sum\limits_{t=1}^k (\ell_t-t)}\right)
\left(\prod_{t=0}^k 2^{-(\ell_{t+1}-\ell_t) i_{\ell_t}}\right) \\
&=\sum_{i_0=0}^\infty \sum_{i_{\ell_1}\ge i_{0}}^\infty \cdots \sum_{i_{\ell_k}\ge i_{\ell_{k-1}}}^\infty  \left( |E_{i_0}|~ 2^{-\ell_1 i_0}\right)  \left( \prod_{t=1}^k |E_{i_{\ell_t}}|~ q^{-\ell_t+t} ~2^{-(\ell_{t+1}-\ell_t) i_{\ell_t}}\right).
\end{align}
Now, we shall observe that for each $t=1,2, \ldots, k,$
\begin{equation}\label{sizeE} |E_{i_{\ell_t}}|~ q^{-\ell_t+t} \le 2^{\frac{(d+1)(d-\ell_t+t)}{d(k+1)}i_{\ell_t}}.\end{equation}
Since $|E_{i_{\ell_t}}|\le q^d~$ (namely, $|E_{i_{\ell_t}}|^{1/d}\le q$), and $\ell_t -t\ge 0$, it is obvious that
$$ |E_{i_{\ell_t}}|^{\frac{\ell_t-t}{d}} \le q^{\ell_t-t} \quad
\mbox{or}\quad |E_{i_{\ell_t}}|^{\frac{\ell_t-t}{d}}q^{-\ell_t+t}\le 1.$$

On the other hand, we see from \eqref{easy3} that
$$ |E_{i_{\ell_t}}|\le 2^{\frac{(d+1) }{k+1}i_{\ell_t}}.$$
Then \eqref{sizeE} is easily shown by observing
\begin{align*}|E_{i_{\ell_t}}|~ q^{-\ell_t+t} &=|E_{i_{\ell_t}}|^{1-\frac{(\ell_t-t)}{d}} |E_{i_{\ell_t}}|^{\frac{(\ell_t-t)}{d}} q^{-\ell_t+t}\\
&\le |E_{i_{\ell_t}}|^{\frac{d-\ell_t+t}{d}} \le 2^{\frac{(d+1)(d-\ell_t+t)}{d(k+1)}i_{\ell_t}}. \end{align*}

From \eqref{allmost} and \eqref{sizeE}, we have
\begin{equation}\label{finalseries} S\lesssim \sum_{i_0=0}^\infty \sum_{i_{\ell_1}\ge i_{0}}^\infty\sum_{i_{\ell_2}\ge i_{\ell_1}}^\infty \cdots
\sum_{i_{\ell_k}\ge i_{\ell_{k-1}}}^\infty  \left( |E_{i_0}|~ 2^{-\ell_1 i_0}\right)  \left( \prod_{t=1}^k 2^{\left(\frac{(d+1)(d-\ell_t+t)}{d(k+1)}-\ell_{t+1}+\ell_t\right) i_{\ell_t}}\right).\end{equation}
Using a simple fact that the value of a convergent geometric series is similar as the first term of the series,
we shall repeatedly compute the inner sums
\begin{equation}\label{Se} \mbox{I}:= \sum_{i_{\ell_1}\ge i_{0}}^\infty\sum_{i_{\ell_2}\ge i_{\ell_1}}^\infty \cdots
\sum_{i_{\ell_k}\ge i_{\ell_{k-1}}}^\infty \left( \prod_{t=1}^k 2^{\left(\frac{(d+1)(d-\ell_t+t)}{d(k+1)}-\ell_{t+1}+\ell_t\right) i_{\ell_t}}\right)\end{equation}
from the variable $i_{\ell_k}$ to the variable $i_{\ell_1}.$
However, to repeatly compute the inner sums we must make sure that each geometric series converges.
To assert that each series is convergent, it will be enough to show that for every $r=1,2, \ldots, k,$
\begin{equation}\label{covc} \sum_{t=r}^k \left(\frac{(d+1)(d-\ell_t+t)}{d(k+1)}-\ell_{t+1}+\ell_t \right) <0.\end{equation}
Now le us see why  \eqref{covc} holds.
Multiplying \eqref{covc} by the factor $d(k+1),$ we see that  the statement \eqref{covc} is same as
$$\sum_{t=r}^k \left((d+1)(d-\ell_t+t) + d(k+1)(\ell_t-\ell_{t+1}) \right) <0.$$
Since $\sum\limits_{t=r}^k d =d(k-r+1)$ and $\sum\limits_{t=r}^k (\ell_t-\ell_{t+1})=\ell_r -\ell_{k+1} = \ell_r -(d+1),$ the above condition is equivalent to
$$ (d+1) \left( d(k-r+1) + \sum_{t=r}^k (t-\ell_t) \right) + d(k+1)\ell_r-d(d+1)(k+1) <0.$$
Write $d(k+1)\ell_r= d(k+1) (\ell_r-r)+ dr(k+1)$ and try to simplify the left hand side of the above inequality. Then, for $r=1,2,\ldots,k$,  we can easily see  that  the above inequality becomes
$$  dr(k-d) + (1-dk)(r-\ell_r) + (d+1) \sum_{t=r+1}^k (t-\ell_t) <0,$$
where we assume that  $\sum_{t=r+1}^k (t-\ell_t)=0$ if $k=1.$
This condition is clearly same as
\begin{equation}\label{reduce11} (dk-1)(\ell_r-r) < dr(d-k)+(d+1) \sum_{t=r+1}^k (\ell_t-t).\end{equation}
To prove this equality, let $\alpha=\ell_r-r\ge 0.$ Since  $\ell_0, \ell_1, \ldots, \ell_k$ are nonnegative integers with $0=\ell_0 < \ell_1<\ell_2<\cdots < \ell_k\le d,$ it is clear that
$\alpha=\ell_r-r \le \ell_t-t$ for all $t\ge r.$ Therefore, to prove \eqref{reduce11}, it will be enough to show that for $r=1,2, \ldots,k,$
\begin{equation}\label{mamuri} (dk-1)\alpha < dr(d-k)+(d+1)(k-r) \alpha.$$
Soving for $\alpha$, this inequality is equivalent to
$$ \alpha < \frac{dr(d-k)}{dr-k+r-1}.\end{equation}
Now, observe that for each $r=1,2,\ldots, k,$ the maximum value of $\ell_r$ happens in the case when
$\ell_k=d,~ \ell_{k-1}=d-1, ~\ell_{k-2}=d-2, \ldots, \ell_{r}=d-k+r.$ This implies that  $\alpha=\ell_r-r\le d-k.$
Hence, to prove \eqref{mamuri}, it suffices to show that for $r=1,2,\ldots, k<d,$
$$ d-k < \frac{dr(d-k)}{dr-k+r-1},$$
which is equivalent to the inequality $r<k+1.$
Since $r=1,2, \ldots, k,$ this inequality clearly holds. This proves \eqref{covc} which implies that  each of  inner sums \eqref{Se} is a convergent geometric series whose value is similar to its first term.
Computing the sum $\mbox{I}$ in \eqref{Se} by this fact, we have
$$ \mbox{I}\sim 2^{i_0\left( \sum\limits_{t=1}^k \frac{(d+1)(d-\ell_t+t)}{d(k+1)}-\ell_{t+1}+\ell_t\right)}$$
As before, since $\ell_{k+1}=d+1$ and $ \sum\limits_{t=1}^k (\ell_t-\ell_{t+1})=\ell_1-\ell_{k+1},$ we can check that
\begin{align*}\sum_{t=1}^k \left(\frac{(d+1)(d-\ell_t+t)}{d(k+1)}-\ell_{t+1}+\ell_t \right)
&= -\frac{(d+1)}{k+1}+\ell_1 + \frac{d+1}{d(k+1)} ~\sum_{t=1}^k (t-\ell_t)\\
&\le -\frac{(d+1)}{k+1}+\ell_1. \end{align*}
Hence, we see that
$$ \mbox{I} \lesssim 2^{i_0 \left(-\frac{(d+1)}{k+1}+\ell_1\right)}.$$
Recall the definition of $\mbox{I}$ in \eqref{Se}.
Then Combining this estimate with \eqref{finalseries} yields
$$ S \lesssim \sum_{i_0=0}^\infty |E_{i_0}|2^{-\frac{(d+1)}{k+1}i_0} =1,$$
where the equality follows from \eqref{easy3}. Thus, we finish the proof.


\begin{thebibliography}{7}


\bibitem{Bu} J. Bueti, \emph{An incidence bound for $k$-planes in $\mathbb F^n$ and a planar variant of the Kakeya maximal function}, preprint.

\bibitem{BCW} J. Bennett, A. Carbery, and J. Wright, \emph{A non-linear Loomis-Whitney inequality and applications}, Math. Res. Letters, {\bf 12} (2005), no.4, 443-457. 

 \bibitem{Ca} A. Carbery, \emph{A multilinear generalisation of the Cauchy-Schwarz inequalitiy}, Proc. Amer. Math. Soc. {\bf 132} (2004),  no.11, 3141-3152. 

\bibitem{Ch84} M. Christ, {\it Estimates for the $k$-plane transform,} Indiana Univ. Math. J. {\bf 33} (1984), 891-910.


\bibitem{CSW08} A.~Carbery, B.~Stones, and J.~Wright, \emph{Averages in vector spaces over finite fields,} Math. Proc. Camb. Phil. Soc. {\bf 144}  (2008),  no.1, 13-27.

\bibitem{Dr84} S.W. Drury, {\it Generalizations of Riesz potentials and $L^p$ estimates for certain $k$-plane transforms,} Illinois J. Math. {\bf 28} (1984), 495-512.

\bibitem{Dv} Z.~ Dvir, \emph{On the size of Kakeya sets in finite fields}, J. Amer. Math. Soc. {\bf 22} (2009), 1093-1097.

\bibitem{EOT} J. S. Ellenberg,~R. Oberlin, and T.~ Tao, \textit{The Kakeya set and maximal conjectures for algebraic varieties over finite fields},  Mathematika, {\bf 56}(2009), no.1,  1-25.

\bibitem{IK10} A. Iosevich and D. Koh, \emph{Extension theorems for spheres in the finite field setting}, Forum. Math. {\bf 22} (2010), no.3, 457--483.

\bibitem{Ko13} D. Koh, \emph{Sharp endpoint estimates for the $X$-ray transform and the Radon transform in finite fields}, Proc. Amer. Math. Soc. {\bf 141} (2013), no.8, 2799-2808.

\bibitem{Ko16} D. Koh, \emph{Conjecture and improved extension theorems for paraboloids in the finite field setting}, preprint (2016).

\bibitem{KS12} D. ~Koh and C.~ Shen, \emph {Sharp extension theorems and Falconer distance problems for algebraic curves in two dimensional vector spaces over finite fields }, Rev. Mat. Iberoam., {\bf 28} (2012), no.1, 157-178.
\bibitem{LL13}  M.~Lewko, \emph{New restriction estimates for the 3-d paraboloid over finite fields},   Adv. Math. {\bf 270} (2015),  no.1, 457-479.
\bibitem{Le14}  M.~Lewko, \emph{Finite field restriction estimates based on Kakeya maximal operator estimates},  arXiv:1401.8011.

\bibitem{LL10} A.~ Lewko and M.~Lewko, \emph{Endpoint restriction estimates for the paraboloid over finite fields}, Proc. Amer. Math. Soc. {\bf 140} (2012), 2013-2028.

\bibitem{MT04}  G. Mockenhaupt and T. Tao, {\it Restriction and Kakeya phenomena for  finite fields,} Duke Math.J. {\bf 121} (2004), 35-74.

\bibitem{Wo} T.~Wolff, \emph{Recent work connected with the Kakeya problem},  Prospects in mathematics (Princeton, NJ, 1996),  129-162, Amer. Math. Soc., Providence, RI, 1999.

\bibitem{Wr13} J. Wright, \emph{Discrete analogues in harmonic analysis}, Poland Lectures - November 2013,\\
http://www.math.uni.wroc.pl/~jdziuban/poland-lectures-wright.pdf.
\end{thebibliography}
\end{document}